\documentclass{amsart}
\usepackage{graphicx}
\usepackage{amsmath}
\usepackage{amssymb}
\usepackage{amsthm}

\usepackage{fullpage}

\newtheorem{theorem}{Theorem}[section]

\newtheorem{lemma}[theorem]{Lemma}
\newtheorem{proposition}[theorem]{Proposition}
\theoremstyle{definition}

\theoremstyle{remark}

\theoremstyle{definition}

\numberwithin{equation}{section}

\newcommand{\norm}[1]{\left\Vert#1\right\Vert}

\newcommand{\set}[1]{\left\{#1\right\}}

\newcommand{\R}{\mathbb R}

\newcommand{\C}{\mathbb C}
\newcommand{\HH}{\mathcal{H}}
\newcommand{\Ch}{\mathcal{C}}
\newcommand{\MM}{\mathcal{M}}

\newcommand{\vv}{\mathrm{v}}
\newcommand{\eps}{\varepsilon}
\renewcommand{\Im}{\mathrm{Im}}
\renewcommand{\Re}{\mathrm{Re}}
\begin{document}

\title[Dynamics of solitary waves]{Long time dynamics of highly concentrated solitary waves for the nonlinear Schr\"odinger equation}%
\author{Claudio Bonanno}%
\address{Dipartimento di Matematica, Universit\`a di Pisa, Largo Bruno Pontecorvo n.5, 56127 Pisa, Italy}%
\email{bonanno@dm.unipi.it}%

\begin{abstract}
In this paper we study the behavior of solutions of a nonlinear Schr\"odinger equation in presence of an external potential, which is allowed be singular at one point. We show that the solution behaves like a solitary wave for long time even if we start from a unstable solitary wave, and its dynamics coincide with that of a classical particle evolving according to a natural effective Hamiltonian.
\end{abstract}
\maketitle
\section{Introduction and statement of the results} \label{sec:intro}
In this paper we study the long time dynamics of a solitary wave solution of a nonlinear Schr\"odinger equation (NLS) in presence of an external potential. This problem has been considerably studied in the last years, following the tradition of the work on the stability of solitons which dates back to Weinstein \cite{weinstein}. 

The first dynamical results are given in \cite{bj} and improved, along the same lines, in \cite{keraani}. This first approach is purely variational and is based on the non-degeneracy conditions proved in \cite{we2} for the ground state of the elliptic equation solved by the function describing the profile of a soliton. This approach has been used also in \cite{sing}, where the results of \cite{bj,keraani} are extended to the case of a potential with a singularity.

A second line of investigations on our problem has been initiated in \cite{froh1,froh2}. In these papers the authors have strongly used the Hamiltonian nature of NLS, approximating the solution by its symplectic projection on the finite dimensional manifold of solitons (see \eqref{man-eps}, which is a sub-manifold of that used in \cite{froh1,froh2}, since we fix the profile $U$). This approach has been improved in \cite{hz2,hz} for the Gross-Pitaevskii equation by showing that it is possible to obtain an exact dynamics for the center of the soliton approximation. 

In the previous papers the non-degeneracy condition for the ground state is a fundamental assumption. It has been removed in a more recent approach introduced in \cite{bgm1,bgm2}. The idea of these papers is that it is possible for the solution of the NLS to remain concentrated for long time and to have a soliton behavior, even if the profile of the initial condition is degenerate for the energy associated to the elliptic equation. In fact the concentration of the solution follows in the semi-classical regime from the role played by the nonlinear term, which in \cite{bgm1,bgm2} is assumed to be dependent on the Planck constant. This approach has been used in \cite{choq} for the NLS with a Hartree nonlinearity, in which case the non-degeneracy of the ground state is for the moment an open question.

In this paper we put together the last two approaches and try to weaken as much as possible the assumptions on the solitary wave. First of all, one main difference is that we control only the $L^{2}$ norm of the difference between the solution of NLS and the approximating traveling solitary wave. This has been done also in \cite{colliding}, and allows to drop the non-degeneracy condition and consider more general nonlinearities. Moreover we prove that the approximation of the solution of NLS with a traveling solitary wave is good also if the solitary wave is not stable, that is it is not a soliton, and the profile is fixed. This choice partly destroys the symplectic structure used in \cite{froh1} and subsequent papers, but we prove that there exists a particular projection on the manifold $\MM_{\eps}$ defined in \eqref{man-eps} which is almost symplectic for long time. Actually this particular projection is natural, since it is defined in terms of the Hamiltonian functional of NLS restricted to the manifold $\MM_{\eps}$, called the \emph{effective Hamiltonian} in \cite{hz}. Then, this almost symplectic projection is enough to prove that the approximation is good for long time. Finally, we remark that we are able to consider the cases of regular and singular external potentials at the same time, and slightly improve on the range of allowed behavior at the singularity with respect to \cite{sing}. 

In the remaining part of this section we describe the problem and the main result and discuss the assumptions. In Section \ref{sec:ham} we use the Hamiltonian nature of NLS to introduce the effective Hamiltonian on the manifold $\MM_{\eps}$ and to find the ``natural'' projection of the solution on $\MM_{\eps}$. In Sections \ref{sec:approx} and \ref{sec:proof} we describe the approximation of the solution of NLS and prove the main result. Finally in the appendix we show that our projection is almost symplectic for long time.

\subsection{The problem and the assumptions}
We study the behavior of solutions $\psi(t,\cdot) \in H^{1}(\R^{N},\C)$, with $N\ge 3$ to the initial value problem
\begin{equation}\label{prob-eps} 
\left\{
\begin{array}{l}
i \eps \psi_{t} + \eps^{2} \triangle \psi - f(\eps^{-2\alpha}|\psi|^{2}) \psi = V(x) \psi \\[0.2cm]
\psi(0,x) = \eps^{\gamma} U(\eps^{-\beta}(x-a_{0}))\, e^{\frac i\eps (\frac 12 (x-a_{0}) \cdot \xi_{0} + \theta_{0} )} 
\end{array}
\right. \tag{$\mathcal{P}_\eps$}
\end{equation}
where $\eps>0$ represents the Planck constant, $\alpha,\beta,\gamma$ are real parameters, $(a_{0},\xi_{0}) \in \R^{N}\times \R^{N}$ are the initial conditions of the finite dimensional dynamics which the solution follows, $\theta\in \R$ is the phase shift. Moreover $U\in H^{1}(\R^{N})$ is a positive function which satisfies
\begin{equation}\label{ellittica}
-\triangle U + f(U^{2})U - \omega U = 0
\end{equation}
for some $\omega \in \R$, and such that
\begin{itemize}
\item[(C1)] $\norm{U}_{L^{2}}^{2} = \rho>0$;
\item[(C2)] $U(x)$ is in $L^{\infty}(\R^{N})$ and vanishes as $|x|\to \infty$ fast enough so that
\[
\norm{|x|U^{2}}_{L^{1}} +\norm{|x|^{2}U^{2}}_{L^{2}}+ \norm{|x| |\nabla U|}_{L^{2}} < \infty
\]
\end{itemize}
Finally we assume that $f$ satisfies
\begin{itemize}
\item[(N1)] There exists a $C^{3}$ functional $F:H^{1}\to \R$ such that $d(F(|\psi|^{2})) = 2 f(|\psi|^{2})\psi$;
\item[(N2)] if $\varphi \in L^{r}$, for some $r\in (2,\frac{2N}{N-2})$, and $U$ is as above, then there exists $C=C(\varphi,U)>0$ only depending on $\varphi$ and $U$ such that
\[
\Big| \int_{\R^{N}}\, \Big[f(|U+v|^{2})(U+v) - f(U^{2})U - (2f'(U^{2})U^{2}+f(U^{2}))\Re (v) - i\, f(U^{2}) \Im(v) \Big]\, \overline{\varphi}\, dx \Big| \le C(\varphi,U)\, \norm{v}_{L^{2}}
\]
for all $v\in H^{1}$ with $\norm{v}_{L^{2}}\le 1$.
\end{itemize}

In studying the behavior of a solution $\psi(t,x)$ to \eqref{prob-eps}, the potential $V$ is considered as an external perturbation and, when $V\equiv 0$ we ask the solution to the initial value problem to be a solitary wave traveling along the unperturbed trajectory $a(t) = a_{0} + \xi_{0} t$, $\xi(t) = \xi_{0}$, namely
\[
\psi_{_{V\equiv 0}}(t,x) = \eps^{\gamma} U(\eps^{-\beta}(x-a(t)))\, e^{\frac i\eps (\frac 12 (x-a(t)) \cdot \xi(t) + \theta(t) )}
\]
Using this expression for $\psi$ in \eqref{prob-eps} with $V\equiv 0$, we obtain an identity if
\[
- \eps^{2-2\beta} \triangle U + f(\eps^{2(\gamma-\alpha)} U^{2})U + \Big(\dot \theta(t) - \frac 14 |\xi_{0}|^{2}\Big) U =0
\]
Using \eqref{ellittica} this implies that either
\begin{equation}\label{par1}
\alpha = \gamma\, , \quad \beta=1\quad \text{and} \quad \theta(t) = \Big(\frac 14 |\xi_{0}|^{2}- \omega \Big) t \, ,
\end{equation}
or we assume that 
\begin{itemize}
\item[(N3)] $f$ is homogeneous of degree $p\in (0, \frac{2}{N-2})$;
\item[(N4)] $\beta = 1 + (\alpha - \gamma) p$ 
\end{itemize}
and $\theta(t) = \Big(\frac 14 |\xi_{0}|^{2} - \eps^{2-2\beta} \omega \Big) t$.

Notice that \eqref{par1} is a particular case of condition (N4), for which we don't need (N3). So in the sequel we assume (N3) and (N4) with the warning that (N3) is not necessary if the particular condition \eqref{par1} holds.

Finally for what concerns the potential $V$ we consider two possible cases:
\begin{itemize}
\item[(Vr)] $V:\R^{N}\to \R$ is a $C^{2}$ function which is bounded from below and with bounded second derivatives, namely 
\[
h_{V}:= \sup_{x} \Big( \max_{i,j} \Big| \frac{\partial^{2}V}{\partial x_{i} \partial x_{j}}(x)\Big| \Big) < \infty\, ;
\]
\item[(Vs)] $V$ is singular at $x=0$ and satifisfies
\begin{itemize}
\item[(Vs1)] $V:\R^{N}\setminus \{0\} \to \R$ is of class $C^{2}$;
\item[(Vs2)] $|V(x)| \sim |x|^{-\zeta}$ and $|\nabla V(x)| \lesssim |x|^{-\zeta -1}$ as $|x|\to 0$ for some $\zeta \in (0,2)$; 
\item[(Vs3)] $V\in L^{m}(\{|x|\ge 1\})$ for $m>\frac N2$ and $|\nabla V(x)|\to 0$ as $|x|\to \infty$.
\end{itemize}
\end{itemize}

\subsection{The main result}
In this paper we prove that

\begin{theorem}\label{main-result}
Let $U$ be a positive solution of \eqref{ellittica} for some $\omega \in \R$ and satisfying (C1) and (C2). Let $f$ satisfy (N1)-(N3) and $\alpha, \beta, \gamma \in \R$ satisfy (N4) and assume that
$\beta \ge 1$, $\alpha\ge \gamma\ge 0$ and, if $N=3$
\begin{equation}\label{serve3}
\beta < 2\gamma +2 ,
\end{equation}
and no further assumption if $N\ge 4$. Then we have
\[
\delta := 1 + \gamma + \beta\left(\frac N2 -2 \right)> 0 \, .
\]
Let $(a(t), \xi(t), \vartheta(t))$ be the solution of the system
\[
\left\{
\begin{aligned}
& \dot a = \xi \\
& \dot \xi = - \frac 2\rho\, \int_{\R^{N}} \nabla V(a+\eps^{\beta}x) U^{2}(x)\, dx\\
& \dot \vartheta = \frac 14 |\xi(t)|^{2} - \eps^{2-2\beta}\, \omega - V(a(t))
\end{aligned}
\right.
\]
with initial condition $(a_{0},\xi_{0},\theta_{0})$, and assume that $V$ satisfies (Vr) or (Vs). If $V$ satisfies (Vs) we also assume that $(a_{0},\xi_{0},\theta_{0})$ is such that $\min_{t} |a(t)| = \bar a>0$.
\vskip 0.2cm
\indent If the solution $\psi(t,x)$ to \eqref{prob-eps} exists for all $t\in \R$, then we can write
\begin{equation} \label{forma}
\psi(t,x) = \eps^{\gamma}\, U(\eps^{-\beta}(x-a(t))) \, e^{\frac i\eps ( \frac 12 (x-a(t))\cdot \xi(t) + \vartheta(t)} + w(t,x)\, e^{\frac i\eps ( \vartheta(t)-\theta_{0})}
\end{equation}
where for any fixed $\eta \in (0,\delta)$
\begin{equation} \label{errore}
\norm{w(t,x)\, e^{\frac i\eps ( \vartheta(t)-\theta_{0})}}_{L^{2}} = O(\eps^{\eta})
\end{equation}
for all $t\in (0,T)$ with $T=O(\eps^{\eta-\delta})$.
\end{theorem}

We now comment on the results of the theorem, in particular with respect to the values of the parameters $\alpha,\beta,\gamma$. We notice that if $\psi(t,x)$ is a solution to \eqref{prob-eps}, then $\tilde \psi(t,x) := \psi(\eps t, \eps x)$ is a solution to
\begin{equation}\label{prob-uno} 
\left\{
\begin{array}{l}
i \tilde \psi_{t} + \triangle \tilde \psi - f(\eps^{-2\alpha}|\tilde\psi|^{2}) \tilde\psi = V(\eps x) \tilde\psi \\[0.2cm]
\tilde \psi(0,x) = \eps^{\gamma} U(\eps^{-\beta+1}x)\, e^{\frac i\eps (\frac 12 x \cdot \xi_{0} + \theta_{0} )} 
\end{array}
\right. \tag{$\mathcal P_{1}$}
\end{equation}
where we set $a_0=0$ for simplicity. Under the same assumptions for $U$, $F$ and $V$ of Theorem \ref{main-result}, we obtain that $\tilde \psi(t,x)$ can still be written as in \eqref{forma}, but now for any fixed $\eta \in (0,\delta)$, the estimate of the error is of the order
\begin{equation}\label{forma2}
\norm{w(t,x)}_{L^{2}} = O(\eps^{\eta-\frac N2})
\end{equation}
for all $t\in (0,\tilde T)$ with $\tilde T= O(\eps^{\eta-\delta-1})$. Hence the time of validity of the approximation has increased by a factor $\eps^{-1}$, but for the estimate \eqref{forma2} to make sense we need that
\begin{equation} \label{aiuto}
\delta > \frac N2 \quad \Leftrightarrow \quad 
\left\{ 
\begin{aligned} 
& \beta < 2\gamma -1\, , & \text{if $N=3$;} \\ 
& \beta \left( \frac N2 -2 \right) + \gamma > \frac N2 -1\, , & \text{if $N\ge 4$.} 
\end{aligned} 
\right.
\end{equation}
Notice in particular that if $\alpha$ is big enough, we can choose $\beta$ and $\gamma$ satisfying \eqref{serve3} and \eqref{aiuto}. So the bigger the enhancement of the nonlinear term the better the approximation of the solution in \eqref{errore}.

Finally, if $\psi(t,x)$ is a solution to \eqref{prob-eps}, then $\eps^{-\alpha} \psi(t,x)$ is a solution to
\begin{equation} \label{prob}
\left\{
\begin{array}{l}
i \eps \psi_{t} + \eps^{2}\triangle \psi - f(|\psi|^{2}) \psi = V(x)\psi \\[0.2cm]
\psi(0,x) = \eps^{\gamma-\alpha} U(\eps^{-\beta}(x-a_{0}))\, e^{\frac i\eps (\frac 12 (x-a_{0}) \cdot \xi_{0} + \theta_{0} )} 
\end{array}
\right. \tag{$\mathcal{P}$}
\end{equation}
Hence under the assumptions of Theorem \ref{main-result}, for any $\eta\in (0,\delta)$ solutions of \eqref{prob} can be written as in \eqref{forma} up to times $T=O(\eps^{\eta-\delta})$.

\subsection{Remarks on the assumptions} We briefly discuss the assumptions on:

\underline{The solitary wave $U$.} The positive function $U$ is the profile of the solitary wave which is the approximation of the solution of \eqref{prob-eps}. Typically one assumes that $U$ is not only solution of \eqref{ellittica} but the minimizer of the energy
\[
\mathcal{E}(u):= \int_{\R^{N}}\, \Big[\frac 12 |\nabla u|^{2} + F(u^{2})\Big] \, dx
\]
constrained to the manifold $\Sigma_{\rho}:= \set{u\in H^{1}\, :\, \norm{u}_{L^{2}}^{2}= \rho}$. This is useful because the minimizer of $\mathcal{E}$ is orbitally stable (see \cite{gss87} and \cite{bbgm}), and is then called soliton. Conditions sufficient for orbital stability are for example assumed in \cite{froh1}, \cite{hz} and \cite{colliding}. In this paper we only assume that $U$ is a solution of \eqref{ellittica}, that is just a critical point for $\mathcal{E}$ constrained to $\Sigma_{\rho}$, and is not necessarily orbitally stable.

For what concerns assumption (C2), we only need $U\in L^{\infty}$ for the case $V$ singular. The speed of vanishing at infinity is verified for example when $U$ is a ground state, in which case $U$ and $\nabla U$ decay exponentially (\cite{bl,kwong}).

\underline{The nonlinearity $f$.} We first discuss (N2). This assumption is used also in \cite{colliding} to which we refer for the proof that (N2) is satisfied if:
\begin{itemize} 
\item $f$ is a Hartree nonlinearity
\[
f(|\psi|^{2})\psi = (W(x) \star |\psi|^{2})\psi
\] 
with $W$ positive, spherically symmetric, in $L^{q}+L^{\infty}$ with $q> \max\set{\frac N2,2}$, and decaying at infinity;

\item $f$ is a local nonlinearity, that is we can write $f:\R^{+}\to \R$ with $f(s) = F'(s)$ for a $C^{3}$ function $F:\R^{+}\to \R$, and
\[
\sup_{s\in \R^{+}}\, s^{\frac{2k-1}{2}}\, f^{(k)}(s) < \infty\, , \qquad k=1,2
\]
\end{itemize}
Assumption (N3) is satisfied by the Hartree nonlinearities as above with $p=1$ and for  and by power local nonlinearities $f(s) = s^{p}$. However we remark that (N3) is not needed if we assume \eqref{par1}.

\underline{The potential $V$.} The assumptions (Vr) and (Vs) on the potential are needed to have local well-posedness for \eqref{prob-eps} by results in \cite[Chapter 4]{caze}. In particular, concerning (Vs2), local well-posedness is implied by $\zeta<2$ for all $N$. It also implies the finiteness of the Hamiltonian \eqref{ham-nls} and the vector field in \eqref{traj} for which $\zeta<N-1$ is sufficient. Notice that in \cite{sing} it was assumed $\zeta<1$.

\section{Hamiltonian formulation for NLS and the trajectories of the solitary waves} \label{sec:ham}
Following \cite{froh1}, we consider the space $H^{1}(\R^{N},\C)$ equipped with the symplectic form
\[
\omega(\psi,\phi) := \Im \int_{\R^{N}}\, \psi\, \bar\phi \, dx
\]
and problem \eqref{prob-eps} associated to the Hamiltonian functional
\begin{equation}\label{ham-nls}
\HH(\psi) := \frac 12 \int_{\R^{N}}\, \Big[ \eps^{2} |\nabla \psi|^{2} + V(x) |\psi|^{2} + \eps^{2\alpha} F(\eps^{-2\alpha} |\psi|^{2}) \Big] dx
\end{equation}
via the law $\eps \psi_{t} = X_{\HH}$, where $X_{\HH}$ is the vector field satisfying
\[
\omega(\phi, X_{\HH}) = d\HH[\phi] \qquad \forall\, \phi\, .
\]
Since the Hamiltonian $\HH$ is not dependent on time, a solution to \eqref{prob-eps} satisfies $\HH(\psi(t,x)) = \HH(\psi(0,x))$ for all $t$. Moreover the Hamiltonian $\HH$ is invariant under the global gauge transformation $\psi \mapsto e^{i\theta} \psi$ for all $\theta \in \R$. This implies that there exists another conserved quantity for the Hamiltonian flow of $\HH$, and it is given by the charge
\[
\Ch(\psi) := \int_{\R^{N}}\, |\psi|^{2}\, dx\, .
\]
Hence, by assumption (C1), it follows that a solution to \eqref{prob-eps} satisfies
\begin{equation}\label{carica-cost}
\norm{\psi(t,x)}_{L^{2}}^{2} = \eps^{2\gamma+\beta N} \norm{U}_{L^{2}}^{2} = \eps^{2\gamma+\beta N} \rho \qquad \forall\, t\, .
\end{equation}

When $V\equiv 0$, we have seen that the solution belongs to the manifold
\begin{equation}\label{man-eps}
\MM_{\eps}:= \set{ U_{\sigma}(x) := \eps^{\gamma} U( \eps^{-\beta}(x-a)) e^{\frac i\eps (\frac 12 (x-a) \cdot \xi + \theta)}\, \Big/ \, \sigma:= (a,\xi,\theta) \in \R^{N}\times \R^{N}\times \R}
\end{equation}
Following \cite{hz}, we construct an Hamiltonian flow associated to $\HH$ on the manifold $\MM_{\eps}$. To this aim we first have to compute $\Omega_{\sigma}$, the restriction of the symplectic form $\omega$ on $\MM_{\eps}$. The tangent space $T_{_{U_{\sigma}}}\MM_{\eps}$ to $\MM_{\eps}$ in a point $U_{\sigma}$ is generated by
\begin{align}
& z_{j,\sigma}^{\eps}(x):= \frac{\partial U_{\sigma}(x)}{\partial a_{j}} = - \Big(\eps^{\gamma-\beta} \partial_{j} U ( \eps^{-\beta}(x-a)) + i\, \frac{\eps^{\gamma-1}}{2}\, \xi_{j} U( \eps^{-\beta}(x-a))\Big) e^{\frac i\eps (\frac 12 (x-a) \cdot \xi + \theta)},\ \  j=1,\dots,N \label{tang-eps-a}\\
& z_{j,\sigma}^{\eps}(x):= \frac{\partial U_{\sigma}(x)}{\partial \xi_{j}} = i\, \frac{1}{2\eps} x_{j-N}\, U_{\sigma}(x)\, ,\qquad j=N+1,\dots, 2N \label{tang-eps-xi}\\
& z_{2N+1,\sigma}^{\eps}(x):= \frac{\partial U_{\sigma}(x)}{\partial \theta} = i\, \frac{1}{2\eps} \, U_{\sigma}(x) \label{tang-eps-theta}
\end{align}
Hence
\[
\Omega_{\sigma} := \omega(z_{i,\sigma}^{\eps},z_{j,\sigma}^{\eps})_{_{1\le i,j \le 2N+1}} = \frac 14 \eps^{2\gamma + \beta N -1} \rho \left( \begin{array}{ccc} 0_{_{N\times N}} & -I_{_{N\times N}} & 0_{_{1\times N}}\\ I_{_{N\times N}} & 0_{_{N\times N}} & 0_{_{1\times N}} \\ 0_{_{N\times 1}} & 0_{_{N\times 1}} & 0 \end{array} \right) = \frac 14 \eps^{2\gamma + \beta N -1} \rho \, d\xi \wedge da
\]
The form $\Omega_{\sigma}$ is degenerate and so $\MM_{\eps}$ is not a symplectic manifold. This is in contrast to \cite{froh1} and \cite{hz} where the soliton manifold was defined also varying the parameter $\omega$ in \eqref{ellittica}. Anyway we use $\Omega_{\sigma}$ to obtain a dynamical system for $\sigma = (a,\xi,\theta)$ associated to the effective Hamiltonian 
\[
\HH_{\MM}(a,\xi,\theta) := \HH(U_{\sigma}) = \frac 18 \eps^{2\gamma+\beta N} \rho |\xi|^{2} + \frac 12 \eps^{2\gamma + \beta N} \int_{\R^{N}} V(a+\eps^{\beta}x) U^{2}(x)\, dx + const(U)\, .
\]
It follows that
\[
\eps \dot \sigma = \frac 4\rho \eps^{-(2\gamma + \beta N -1)} \left( \begin{array}{c} \partial_{\xi} \HH_{\MM} \\ - \partial_{a} \HH_{\MM} \\ 0 \end{array} \right)
\]
hence the system of differential equations
\begin{equation}\label{traj}
\left\{
\begin{aligned}
& \dot a = \xi \\
& \dot \xi = - \frac 2\rho\, \int_{\R^{N}} \nabla V(a+\eps^{\beta}x) U^{2}(x)\, dx\\
& \dot \theta = 0
\end{aligned}
\right.
\end{equation}
For a solution $\sigma(t) = (a(t), \xi(t), \theta(t))$ of \eqref{traj} we introduce the following notation which is needed below:
\begin{equation}\label{vt}
\vv(t) := \frac 1 \rho\,  \int_{\R^{N}} \Big(\nabla V(a(t)+\eps^{\beta}x) -\nabla V(a(t))\Big)U^{2}(x)\, dx
\end{equation}

\section{Approximation of the solution} \label{sec:approx}

In \cite{froh1} and related papers, the main idea was to prove the existence of a unique symplectic decomposition for the solution of \eqref{prob-eps} up to a given time $\tau$. This was achieved by proving that the solution stays for $t\le \tau$ in a small tubular neighbourhood of the symplectic manifold $\MM_{\eps}$, and using the existence of a symplectic projection on $\MM_{\eps}$. In this paper instead we define a particular projection of the solution on the manifold $\MM_{\eps}$, projection which turns out to be ``almost'' symplectic up to some time $\tau$, and show that difference between the solution and the projection is small for $t\le \tau$.

Let $\psi(t,x)$ be the solution of \eqref{prob-eps} and assume that it is defined for all $t\in \R$. Let $\sigma(t)= (a(t), \xi(t), \theta(t))$ be the solution of \eqref{traj} with initial conditions $\sigma_{0}= (a_{0}, \xi_{0}, \theta_{0})$, and $U_{\sigma(t)}$ the element in $\MM_{\eps}$ associated to $\sigma(t)$. Moreover, let $\omega_{\eps}(t)$ be a solution of the Cauchy problem
\begin{equation}\label{cauchy-om}
\left\{ \begin{aligned} & \dot \omega_{\eps}(t) = \eps^{2-2\beta} \omega - \frac 14 |\xi(t)|^{2}+ V(a(t)) \\ & \omega_{\eps}(0) = 0
\end{aligned}
\right.
\end{equation}
Notice that in the statement of Theorem \ref{main-result} we use the notation $\vartheta(t) = \theta(t) - \omega_{\eps}(t)$.

Then we define
\begin{equation}\label{w}
w(t,x) := e^{\frac i\eps\, \omega_{\eps}(t)}\, \psi(t,x) - U_{\sigma(t)}(x)\, .
\end{equation}
and using $\tilde x = \eps^{-\beta}(x-a(t))$
\begin{equation}\label{w-tilde}
\begin{aligned}
\tilde w(t,\tilde x) := & \eps^{-\gamma} e^{-\frac i\eps (\frac 12 \eps^{\beta}\tilde x \cdot \xi(t) + \theta(t))}\, w(t,a(t) + \eps^{\beta}\tilde x)\\ = & \eps^{-\gamma} e^{-\frac i\eps (\frac 12 \eps^{\beta}\tilde x \cdot \xi(t) + \theta(t) - \omega_{\eps}(t))}\, \psi(t,a(t) + \eps^{\beta}\tilde x) - U(\tilde x)
\end{aligned}
\end{equation}
The functions $w(t,x)$ and $\tilde w(t,\tilde x)$ represent the distance between the solution and the solitary wave, solution with $V\equiv 0$, in the moving and in the fixed space-time frame respectively. Recall from \eqref{tang-eps-a}-\eqref{tang-eps-theta} that the tangent space to the soliton manifold $\MM_{\eps}$ with $\eps=1$ at $\sigma=0$ is generated by
\begin{align}
&z_{j,0}(x) = -\partial_{j}U(x)\, , \qquad j=1,\dots,N \label{tang-a-0}\\
&z_{j,0}(x) = i \frac 12\, x_{j-N} U(x)\, , \qquad j=N+1,\dots,2N \label{tang-xi-0}\\
&z_{2N+1,0}(x) = i U(x)\, . \label{tang-theta-0}
\end{align}
Then
\begin{lemma}\label{sympl-w-tilde-w}
For all $t\in \R$ and all $\sigma=(a,\xi,\theta) \in \R^{2N+1}$, we have
\[
\omega( w, z_{j,\sigma}^{\eps}) = \left\{
\begin{aligned}
& \eps^{2\gamma+\beta (N-1)}\, \omega( \tilde w, z_{j,0}) - \frac 12 \eps^{2\gamma + \beta N-1} \, \xi_{j}(t)\, \omega( \tilde w, z_{2N+1,0})\, , & j=1,\dots,N \\
& \eps^{2\gamma+\beta (N+1)-1}\, \omega( \tilde w, z_{j,0}) + \frac 12 \eps^{2\gamma + \beta N-1} \, a_{j-N}(t)\, \omega( \tilde w, z_{2N+1,0})\, , & j=N+1,\dots,2N\\ 
& \eps^{2\gamma+\beta N-1}\, \omega( \tilde w, z_{j,0})\, , & j=2N+1
\end{aligned}
\right.
\]
\end{lemma}

\begin{proof}
First of all, by standard manipulations, for all $j=1,\dots,2N+1$
\[
\omega( w(t,x), z_{j,\sigma}^{\eps}(x)) = \Im \int_{\R^{N}}\, w(t,x)\, \overline{z_{j,\sigma}^{\eps} (t,x)}\, dx =\eps^{\beta N} \Im \, \int_{\R^{N}}\, w(t,a(t)+\eps^{\beta}\tilde x)\, \overline{z_{j,\sigma}^{\eps} (t,a(t)+\eps^{\beta}\tilde x)}\, d\tilde x =
\]
\[
= \eps^{\beta N} \Im \, \int_{\R^{N}}\, e^{-\frac i\eps (\frac 12 \eps^{\beta}\tilde x \cdot \xi(t) + \theta(t))} w(t,a(t)+\eps^{\beta}\tilde x)\ \overline{e^{-\frac i\eps (\frac 12 \eps^{\beta}\tilde x \cdot \xi(t) + \theta(t))} z_{j,\sigma}^{\eps} (t,a(t)+\eps^{\beta}\tilde x)}\, d\tilde x =
\]
\[
= \eps^{\gamma + \beta N}\, \omega\Big(\tilde w(t,\tilde x), e^{-\frac i\eps (\frac 12 \eps^{\beta}\tilde x \cdot \xi(t) + \theta(t))} z_{j,\sigma}^{\eps} (t,a(t)+\eps^{\beta}\tilde x)\Big)\, ,
\]
where in the last equality we have used \eqref{w-tilde}. Moreover, using \eqref{tang-eps-a}-\eqref{tang-eps-theta} and \eqref{tang-a-0}-\eqref{tang-theta-0}, we have
\[
e^{-\frac i\eps (\frac 12 \eps^{\beta}\tilde x \cdot \xi(t) + \theta(t))} z_{j,\sigma}^{\eps} (t,a(t)+\eps^{\beta}\tilde x) = \left\{
\begin{aligned}
& \eps^{\gamma-\beta} z_{j,0}(\tilde x) - \frac 12 \eps^{\gamma-1}\, \xi_{j}(t) z_{2N+1,0}(\tilde x)\, , & j=1,\dots,N\\
& \eps^{\gamma+\beta-1} z_{j,0}(\tilde x) + \frac 12 \eps^{\gamma-1}\, a_{j-N}(t) z_{2N+1,0}(\tilde x)\, , & j=N+1,\dots,2N\\
& \eps^{\gamma -1}\, z_{j,0}(\tilde x)\, , & j= 2N+1
\end{aligned}
\right.
\]
and the proof is finished.
\end{proof}

We first study the evolution in time of the function $\tilde w(t,\tilde x)$.

\begin{proposition}\label{der-w-tilde}
Let $\tilde w(t,\tilde x)$ be defined as in \eqref{w-tilde} for all $t\in \R$, then
\begin{align*}
\partial_{t} \tilde w (t,\tilde x) = & \frac i \eps \Big( \eps^{\beta}\tilde x \cdot \vv(t) - \mathcal{R}_{V}(t,\tilde x)\Big) \, \Big(U(\tilde x) + \tilde w(t,\tilde x)\Big) +\\
& + i \eps^{1-2\beta} \Big[ \triangle \tilde w(t,\tilde x) + \omega\, \tilde w(t,\tilde x) - \Big( 2 f'(U^{2}(\tilde x)) U^{2}(\tilde x) + f(U^{2}(\tilde x)) \Big) \Re (\tilde w(t,\tilde x)) + \\
& - i\, f(U^{2}(\tilde x)) \Im (\tilde w(t,\tilde x))- \mathcal{R}_{F}(t,\tilde x)\Big] 
\end{align*} 
where $\vv(t)$ is defined in \eqref{vt} and
\begin{align*}
& \mathcal{R}_{V}(t,\tilde x):= V(a(t)+\eps^{\beta}\tilde x) -V(a(t)) - \eps^{\beta}\tilde x \cdot \nabla V(a(t))\\
& \mathcal{R}_{F}(t,\tilde x):= f(|U(\tilde x) + \tilde w(t,\tilde x)|^{2})\Big( U(\tilde x) + \tilde w(t,\tilde x) \Big) - f(U^{2}(\tilde x))U(\tilde x) +\\
&  - \Big( 2 f'(U^{2}(\tilde x)) U^{2}(\tilde x) + f(U^{2}(\tilde x)) \Big) \Re (\tilde w(t,\tilde x)) - i\, f(U^{2}(\tilde x)) \Im (\tilde w(t,\tilde x))
\end{align*}
\end{proposition}
\begin{proof}
From \eqref{w-tilde}, we have
\[
\psi(t,x) = \eps^{\gamma}\, e^{\frac i\eps (\frac 12 (x-a(t)) \cdot \xi(t) + \theta(t) - \omega_{\eps}(t))}\, \Big(U(\eps^{-\beta}(x-a(t))) + \tilde w(t, \eps^{-\beta}(x-a(t)))\Big)
\]
hence, using the notation
\[
g(t,x) := e^{\frac i\eps (\frac 12 (x-a(t)) \cdot \xi(t) + \theta(t) - \omega_{\eps}(t))}
\]
we have
\begin{align*}
\partial_{t} \psi(t,x) = & \frac i \eps\, \Big(-\frac 12 \dot a(t)\cdot \xi(t) + \frac 12 (x-a(t)) \cdot \dot \xi(t) + \dot \theta(t) - \dot \omega_{\eps}(t)\Big) \psi(t,x) + \\ 
& + \eps^{\gamma}\, g(t,x)\, \Big[ \partial_{t} \tilde w (t, \eps^{-\beta}(x-a(t))) - \eps^{-\beta} \dot a(t) \cdot \Big(\nabla U (\eps^{-\beta}(x-a(t))) + \nabla \tilde w(t, \eps^{-\beta}(x-a(t)))\Big) \Big]
\end{align*}
Hence
\begin{align*}
\partial_{t} \tilde w (t, \eps^{-\beta}(x-a(t))) = & \eps^{-\gamma} g^{-1}(t,x) \partial_{t} \psi(t,x) + \eps^{-\beta} \dot a(t) \cdot \Big(\nabla U (\eps^{-\beta}(x-a(t))) + \nabla \tilde w(t, \eps^{-\beta}(x-a(t)))\Big)+\\
& + \frac i \eps\, \Big(\frac 12 \dot a(t)\cdot \xi(t) - \frac 12 (x-a(t)) \cdot \dot \xi(t) - \dot \theta(t) + \dot \omega_{\eps}(t)\Big) \eps^{-\gamma} g^{-1}(t,x) \psi(t,x)
\end{align*}
At this point we use that $\psi(t,x)$ is a solution of \eqref{prob-eps} and change variable $\tilde x = \eps^{-\beta}(x-a(t))$, to obtain
\begin{align*}
\partial_{t} \tilde w (t, \tilde x) = & -\frac i \eps \Big( V(a(t) + \eps^{\beta}\tilde x) + f(\eps^{2\gamma-2\alpha} |U(\tilde x) + \tilde w(t,\tilde x)|^{2}) \Big)\, \Big( U(\tilde x) + \tilde w(t,\tilde x) \Big) +\\
& + i \eps^{1-2\beta} \Big( \triangle U(\tilde x) + \triangle \tilde w(t,\tilde x)\Big) - \frac i 4 \eps^{-1} |\xi(t)|^{2} \Big( U(\tilde x) + \tilde w(t,\tilde x) \Big) +\\
& + \eps^{-\beta} \dot a(t) \cdot \Big(\nabla U (\tilde x) + \nabla \tilde w(t, \tilde x)\Big)- \eps^{-\beta}\xi(t) \cdot \Big( \nabla U(\tilde x) + \nabla \tilde w(t,\tilde x) \Big) + \\
& + \frac i \eps\, \Big(\frac 12 \dot a(t)\cdot \xi(t) - \frac 12 \eps^{\beta}\tilde x \cdot \dot \xi(t) - \dot \theta(t) + \dot \omega_{\eps}(t)\Big) \Big( U(\tilde x) + \tilde w(t,\tilde x) \Big)
\end{align*}
We now use that $\sigma(t)=(a(t),\xi(t),\theta(t))$ is a solution of \eqref{traj} and $\omega_{\eps}(t)$ satisfies \eqref{cauchy-om} to see that
\[
\frac 12 \dot a(t)\cdot \xi(t) - \dot \theta(t) + \dot \omega_{\eps}(t) -\frac 14 |\xi(t)|^{2} = \eps^{2-2\beta}\, \omega + V(a(t))
\]
Moreover using that $U$ solves \eqref{ellittica} we have
\begin{align*}
\partial_{t} \tilde w (t, \tilde x) = & \frac i \eps \Big( \frac 1\rho \eps^{\beta}\tilde x\cdot \int_{\R^{N}}\, \nabla V (a(t) + \eps^{\beta} y) U^{2}( y)\, dy + V(a(t)) - V(a(t)+\eps^{\beta}\tilde x)\Big) \Big( U(\tilde x) + \tilde w(t,\tilde x) \Big) +\\
& + \frac i \eps \Big( \eps^{2-2\beta} f(U^{2}(\tilde x))U(\tilde x) - f(\eps^{2\gamma-2\alpha} |U(\tilde x) + \tilde w(t,\tilde x)|^{2})\Big( U(\tilde x) + \tilde w(t,\tilde x) \Big)\Big)+\\
& + i\eps^{1-2\beta} \Big( \triangle \tilde w(t,\tilde x) + \omega\, \tilde w(t,\tilde x)\Big)
\end{align*}
Finally, in the first line we use
\[
\frac 1\rho \eps^{\beta}\tilde x\cdot \int_{\R^{N}}\, \nabla V (a(t) + \eps^{\beta} y) U^{2}( y)\, dy + V(a(t)) - V(a(t)+\eps^{\beta}\tilde x) = \eps^{\beta}\tilde x \cdot \vv(t) - \mathcal{R}_{V}(t,\tilde x)
\]
Then in the second line we write
\begin{align*}
& f(\eps^{2\gamma-2\alpha} |U(\tilde x) + \tilde w(t,\tilde x)|^{2})\Big( U(\tilde x) + \tilde w(t,\tilde x) \Big) =  f(\eps^{2\gamma-2\alpha} U^{2}(\tilde x))U(\tilde x) +\\
&  + \Big[ 2 \eps^{2\gamma-2\alpha}  f'(\eps^{2\gamma-2\alpha} U^{2}(\tilde x)) U^{2}(\tilde x) + f(\eps^{2\gamma-2\alpha} U^{2}(\tilde x)) \Big] \Re(\tilde w(t,\tilde x)) + i\, f(\eps^{2\gamma-2\alpha} U^{2}(\tilde x)) \Im(\tilde w(t,\tilde x)) + \\
& + r_{_{F}}(t,\tilde x)
\end{align*}
where $r_{_{F}}(t,\tilde x)$ is defined by this equality. Using now assumption (N3) we have
\begin{align*}
& f(\eps^{2\gamma-2\alpha} U^{2}(\tilde x))U(\tilde x) = \eps^{2(\gamma-\alpha)p}\, f(U^{2}(\tilde x))U(\tilde x)\\
& \eps^{2\gamma-2\alpha}  f'(\eps^{2\gamma-2\alpha} U^{2}(\tilde x)) = \eps^{2(\gamma-\alpha)p}\, f'(U^{2}(\tilde x))
\end{align*}
and by assumption (N4) $2-2\beta = 2(\gamma-\alpha)p$. Hence
\begin{align*}
& \eps^{2-2\beta} f(U^{2}(\tilde x))U(\tilde x) - f(\eps^{2\gamma-2\alpha} |U(\tilde x) + \tilde w(t,\tilde x)|^{2})\Big( U(\tilde x) + \tilde w(t,\tilde x) \Big) = \\
& = - \eps^{2-2\beta} \Big[ \Big(2 f'(U^{2}(\tilde x)) U^{2}(\tilde x) + f(U^{2}(\tilde x)) \Big) \Re(\tilde w(t,\tilde x)) + i\, f(U^{2}(\tilde x)) \Im(\tilde w(t,\tilde x)) + \mathcal{R}_{F}(t,\tilde x)\Big]
\end{align*}
and the proof is finished.
\end{proof}

We now use Lemma \ref{sympl-w-tilde-w} and Proposition \ref{der-w-tilde} to estimate the growth of the function $w(t,x)$ in $L^{2}$ norm.

\begin{theorem}\label{main-part-1}
Let $\psi(t,x)$ be the solution of \eqref{prob-eps} assumed to be defined for all $t\in \R$. Let $\sigma(t)= (a(t), \xi(t), \theta(t))$ be the solution of \eqref{traj} with initial conditions $\sigma_{0}= (a_{0}, \xi_{0}, \theta_{0})$, and $U_{\sigma(t)}$ the element in $\MM_{\eps}$ associated to $\sigma(t)$. Finally let $\omega_{\eps}(t)$ be the solution of the Cauchy problem \eqref{cauchy-om}. If the function $w(t,x)$ defined in \eqref{w} satisfies $\norm{w(t,\cdot)}_{L^{2}}\le 1$ for $t\in (0, \tau)$ then
\begin{equation}\label{stima-princ-1}
\Big| \partial_{t} \norm{w(t,\cdot)}_{L^{2}}^{2} \Big| \le 4\, \eps^{\gamma +\beta \frac N2}\, \norm{w(t,\cdot)}_{L^{2}} \Big( \frac 12 \eps^{\beta-1} \norm{|x| U(x)}_{L^{2}}\, |\vv(t)| + \eps^{-1} \norm{|\mathcal{R}_{V}(t,x)|\, U(x)}_{L^{2}} + \eps^{1-2\beta}\, C(z_{j,0},U)\Big)
\end{equation}
for all $t\in (0, \tau)$, where $\vv(t)$ is defined in \eqref{vt}, $\mathcal{R}_{V}(t,x)$ is as in Proposition \ref{der-w-tilde}, and $C(z_{j,0},U)$ is defined in (N2).
\end{theorem}
\begin{proof}
From \eqref{w} we write
\[
\norm{e^{\frac i\eps\, \omega_{\eps}(t)} \psi(t,\cdot)}_{L^{2}}^{2} = \norm{w(t,\cdot)}_{L^{2}}^{2} + \norm{U_{\sigma(t)}(\cdot)}_{L^{2}}^{2} + 2\, \Re \int_{\R^{N}}\, w(t,x)\, \overline{U_{\sigma(t)}(x)}\, dx\, .
\]
Moreover using \eqref{tang-eps-theta}
\[
\Re \int_{\R^{N}}\, w(t,x)\, \overline{U_{\sigma(t)}(x)}\, dx = - \Im \int_{\R^{N}}\, w(t,x)\, \overline{i\, U_{\sigma(t)}(x)}\, dx = -2\eps \, \omega(w, z_{2N+1,\sigma}^{\eps})
\]
and using \eqref{carica-cost}
\[
\norm{e^{\frac i\eps\, \omega_{\eps}(t)} \psi(t,\cdot)}_{L^{2}}^{2} = \norm{U_{\sigma(t)}(\cdot)}_{L^{2}}^{2}= \eps^{2\gamma +\beta N}\, \rho\, , \qquad \forall\, t\, .
\]
Hence
\[
\partial_{t} \norm{w(t,\cdot)}_{L^{2}}^{2} = 4\eps \, \partial_{t}\, \Big( \omega(w, z_{2N+1,\sigma}^{\eps})\Big)
\]
We now use Lemma \ref{sympl-w-tilde-w} for $j=2N+1$ to write
\begin{equation} \label{primo-passo}
\partial_{t} \norm{w(t,\cdot)}_{L^{2}}^{2} = 4\eps^{2\gamma+\beta N} \, \partial_{t}\, \Big( \omega(\tilde w, z_{2N+1,0})\Big) = 4\eps^{2\gamma+\beta N} \, \omega(\partial_{t} \tilde w, z_{2N+1,0})
\end{equation}
The final step is to use the results in Appendix \ref{app:calcoli}. In particular, notations \eqref{di1}-\eqref{di4} and Lemmas \ref{i1}-\ref{i4}, imply
\[
\omega(\partial_{t} \tilde w, z_{2N+1,0}) = \omega(I_{2}, z_{2N+1,0}) + \omega(I_{4}, z_{2N+1,0})
\]
and
\[
\Big| \omega(\partial_{t} \tilde w, z_{2N+1,0}) \Big| \le \norm{\tilde w}_{L^{2}} \Big( \frac 12 \eps^{\beta-1} \norm{|\tilde x| U(\tilde x)}_{L^{2}}\, |\vv(t)| + \eps^{-1} \norm{|\mathcal{R}_{V}(t,\tilde x)|\, U(\tilde x)}_{L^{2}} + \eps^{1-2\beta}\, C(z_{j,0},U)\Big)
\]
This, together with \eqref{primo-passo} and
\[
\norm{\tilde w}_{L^{2}} = \eps^{-\gamma -\beta \frac N2}\, \norm{w}_{L^{2}}
\]
which follows from \eqref{w-tilde}, imply \eqref{stima-princ-1}.
\end{proof}

\section{Proof of Theorem \ref{main-result}} \label{sec:proof}

We first study the behavior of $\vv(t)$ and $\mathcal{R}_{V}(t,x)$ as defined in \eqref{vt} and Proposition \ref{der-w-tilde}. Notice that they are defined only in terms of $U$ and $V$, and do not depend on the solution $\psi(t,x)$ of \eqref{prob-eps}.

Let first consider the case of potentials $V$ satisfying assumptions (Vr). By (C2) it is immediate that system \eqref{traj} can be written as
\[
\left\{
\begin{aligned}
& \dot a = \xi \\
& \dot \xi = -2\, \nabla V(a) + O(\eps^{\beta})\\
& \dot \theta = 0
\end{aligned}
\right.
\]
and for all $t\in \R$
\begin{align}
& |\vv(t)| \le \eps^{\beta}\ N\, \frac{h_{V}}{\rho}\, \int_{\R^{N}}\, |x|\, U^{2}(x)\, dx \label{stima-v-r} \\
& |\mathcal{R}_{V}(t,x)| \le \eps^{2\beta}\, \frac{N^{2}}{2} \, h_{V}\, |x|^{2} \label{stima-R-r}
\end{align}

Let now $V$ satisfy assumptions (Vs), then by (Vs2) and (Vs3) it follows that if the solution of \eqref{traj} satisfies $a(t)\not= 0$ for all $t\in \R$, then 
\begin{align}
& \Big| \int_{\R^{N}}\, V(a(t)+\eps^{\beta}x)\, U^{2}(x)\, dx \Big| < \infty \\
& |\vv(t)| \le |\nabla V(a(t))| + \Big| \int_{\R^{N}}\, \nabla V(a(t)+\eps^{\beta}x)\, U^{2}(x)\, dx \Big| < \infty \label{stima-v-s} \\
& \norm{\, |\mathcal{R}_{V}(t,x)|\, \varphi(x)}_{L^{2}} \le const(\zeta,\varphi) \label{stima-R-s}
\end{align}
for all $t\in \R$ and for $\varphi(x) = U(x)$, $|x|U(x)$, $U(x)|\nabla U(x)|$. 

Let $\psi(t,x)$ be the solution of \eqref{prob-eps}. Let $\sigma(t)= (a(t), \xi(t), \theta(t))$ be the solution of \eqref{traj} with initial conditions $\sigma_{0}= (a_{0}, \xi_{0}, \theta_{0})$ such that $a(t)\not= 0$ for all $t\in \R$ if $V$ is singular at the origin, and $U_{\sigma(t)}$ the element in $\MM_{\eps}$ associated to $\sigma(t)$. Finally let $\omega_{\eps}(t)$ be the solution of the Cauchy problem \eqref{cauchy-om}. Then the function $w(t,x)$ defined in \eqref{w} satisfies
\[
\norm{w(0,x)}_{L^{2}} = 0
\]
Hence \eqref{stima-princ-1} holds for $t\in (0,\tau)$ with $\tau$ small enough.

Moreover from \eqref{stima-v-r}-\eqref{stima-R-r} and \eqref{stima-v-s}-\eqref{stima-R-s} follow that in both cases for $V$, there exists a constant $C(U,a_{0},\xi_{0},\zeta)$ not depending on $t$, such that from \eqref{stima-princ-1} we get
\[
\Big| \partial_{t} \norm{w(t,\cdot)}_{L^{2}} \Big| \le C(U,a_{0},\xi_{0},\zeta)\,  \eps^{\gamma+\beta \frac N2} \Big( \eps^{\beta -1} + \eps^{-1} + \eps^{1-2\beta}\Big)\, ,\qquad  \forall\, t\in (0,\tau)
\]
Since $\beta \ge 1$, it follows
\begin{equation} \label{stima-princ-2}
\Big| \partial_{t} \norm{w(t,\cdot)}_{L^{2}} \Big| \le C(U,a_{0},\xi_{0},\zeta)\,  \eps^{1+\gamma+\beta (\frac N2-2)}\, ,\qquad  \forall\, t\in (0,\tau)
\end{equation}
First of all, by (N4) we can write
\[
\delta:= 1+\gamma+\beta \left(\frac N2-2 \right) = \frac N2 - 1 + \gamma + (\alpha-\gamma)\, p \left( \frac N2 -2\right) 
\]
and $\delta >0$ if $N\ge 4$. If $N=3$ instead we also need to assume
\[
\alpha - \gamma < \frac 2p \left(\gamma + \frac 12 \right)
\]
to have $\delta >0$. However, in both cases we get $\tau = O(\eps^{-\delta})$, hence our argument is consistent.

Moreover, for any fixed $\eta \in (0,\delta)$, estimates \eqref{stima-princ-2} immediately implies
\[
\norm{w(t,\cdot)}_{L^{2}} = O(\eps^{\eta})
\]
for all $t \in (0,T)$ with $T= O(\eps^{\eta-\delta})$. The proof is complete. \qed

\appendix

\section{The approximation of the symplectic projection} \label{app:calcoli}

We have approximated the solution $\psi(t,x)$ of \eqref{prob-eps} by a projection on the manifold $\MM_{\eps}$ of solitons. As stated above, the manifold $\MM_{\eps}$ is not symplectic and the projection $U_{\sigma(t)}$ is not obtained by a symplectic decomposition as in \cite{froh1} and subsequent papers. However we now show that the difference
\[
w(t,x) = e^{\frac i \eps \, \omega_{\eps}(t)}\psi(t,x) - U_{\sigma(t)}(t,x)
\]
is almost symplectic orthogonal to $\MM_{\eps}$ for long time. In particular we show that the quantities $\omega(w,z_{j,\sigma}^{\eps})$ increase slowly. 

By Lemma \ref{sympl-w-tilde-w}, we only need to compute the derivatives
\[
\partial_{t} \omega(\tilde w, z_{j,0}) = \omega(\partial_{t}\tilde w, z_{j,0})
\]
and use \eqref{traj}. We use Proposition \ref{der-w-tilde} and write
\[
\partial_{t}\tilde w = I_{1} + I_{2} + I_{3} + I_{4}
\]
where
\begin{align}
& I_{1}:=  \frac i \eps \Big( \eps^{\beta}\tilde x \cdot \vv(t) - \mathcal{R}_{V}(t,\tilde x)\Big) \, U(\tilde x) \label{di1} \\
& I_{2}:=  \frac i \eps \Big( \eps^{\beta}\tilde x \cdot \vv(t) - \mathcal{R}_{V}(t,\tilde x)\Big) \, \tilde w(t,\tilde x) \label{di2}\\
& I_{3}:= \begin{aligned}&  i \eps^{1-2\beta} \Big[ \triangle \tilde w(t,\tilde x) + \omega\, \tilde w(t,\tilde x) - \Big( 2 f'(U^{2}(\tilde x)) U^{2}(\tilde x) + f(U^{2}(\tilde x)) \Big) \Re( \tilde w(t,\tilde x)) + \\ &  - i\, f(U^{2}(\tilde x)) \Im( \tilde w(t,\tilde x)) \Big] \end{aligned} \label{di3}\\
& I_{4}:=  - i \eps^{1-2\beta} \mathcal{R}_{F}(t,\tilde x) \label{di4}
\end{align}

\begin{lemma}\label{i1} Recalling notation \eqref{vt}, we have
\[
\omega(I_{1}, z_{j,0}) = \left\{ 
\begin{aligned}
& \frac 12 \eps^{\beta-1} \rho \vv_{j}(t) + \frac{1}{2\eps} \int_{\R^{N}}\, \mathcal{R}_{V}(t,\tilde x)\, \partial_{j} (U^{2}) (\tilde x)\, d\tilde x\, , & j=1,\dots,N\\
& 0\, , & j=N+1,\dots, 2N+1
\end{aligned}
\right.
\]
\end{lemma}
\begin{proof}
For $j=1,\dots,N$, using \eqref{tang-a-0} and (C2),
\begin{align*}
\omega(I_{1}, z_{j,0}) = & \Im \int_{\R^{N}}\, \frac i \eps \Big( \eps^{\beta}\tilde x \cdot \vv(t) - \mathcal{R}_{V}(t,\tilde x)\Big) \, U(\tilde x) \, \overline{(- \partial_{j}U(\tilde x))}\, d\tilde x =\\
= & - \frac 12 \eps^{\beta-1} \int_{\R^{N}}\, \tilde x \cdot \vv(t) \, \partial_{j} (U^{2}) (\tilde x)\, d\tilde x + \frac{1}{2\eps} \int_{\R^{N}}\, \mathcal{R}_{V}(t,\tilde x)\, \partial_{j} (U^{2}) (\tilde x)\, d\tilde x =\\
= & \frac 12 \eps^{\beta-1} \int_{\R^{N}}\, U^{2}(\tilde x)\, \partial_{j}(\tilde x \cdot \vv(t))\, d\tilde x + \frac{1}{2\eps} \int_{\R^{N}}\, \mathcal{R}_{V}(t,\tilde x)\, \partial_{j} (U^{2}) (\tilde x)\, d\tilde x = \\
= & \frac 12 \eps^{\beta-1} \rho \vv_{j}(t) + \frac{1}{2\eps} \int_{\R^{N}}\, \mathcal{R}_{V}(t,\tilde x)\, \partial_{j} (U^{2}) (\tilde x)\, d\tilde x
\end{align*}
For $j=N+1,\dots, 2N$, using \eqref{tang-xi-0},
\begin{align*}
\omega(I_{1}, z_{j,0}) = & \Im \int_{\R^{N}}\, \frac i \eps \Big( \eps^{\beta}\tilde x \cdot \vv(t) - \mathcal{R}_{V}(t,\tilde x)\Big) \, U(\tilde x) \, \overline{\Big(i \frac 12 x_{j-N} U(\tilde x)\Big)}\, d\tilde x = 0
\end{align*}
For $j=2N+1$, using \eqref{tang-theta-0},
\begin{align*}
\omega(I_{1}, z_{j,0}) = & \Im \int_{\R^{N}}\, \frac i \eps \Big( \eps^{\beta}\tilde x \cdot \vv(t) - \mathcal{R}_{V}(t,\tilde x)\Big) \, U(\tilde x) \, \overline{\Big(i U(\tilde x)\Big)}\, d\tilde x = 0 \\
\end{align*}
and the proof is finished.
\end{proof}

\begin{lemma}\label{i2} Recalling notation \eqref{vt}, we have
\[
|\omega(I_{2}, z_{j,0})| \le \left\{ 
\begin{aligned}
& \norm{\tilde w}_{L^{2}}\, \Big( \eps^{\beta-1} \norm{|\tilde x| \left|\nabla U(\tilde x)\right|}_{L^{2}}\, |\vv(t)| + \eps^{-1} \norm{\mathcal{R}_{V}(t,\cdot)\, |\nabla U|}_{L^{2}}\Big)  \, , & j=1,\dots,N\\
& \norm{\tilde w}_{L^{2}}\, \Big( \frac 12 \eps^{\beta-1} \norm{|\tilde x|^{2} U(\tilde x)}_{L^{2}}\, |\vv(t)| + \eps^{-1} \norm{\mathcal{R}_{V}(t,\tilde x)\, |\tilde x| U(\tilde x)}_{L^{2}}\Big)  \, , & j=N+1,\dots, 2N\\
& \norm{\tilde w}_{L^{2}}\, \Big( \frac 12 \eps^{\beta-1} \norm{|\tilde x| U(\tilde x)}_{L^{2}}\, |\vv(t)| + \eps^{-1} \norm{\mathcal{R}_{V}(t,\cdot)\, U}_{L^{2}}\Big)  \, , & j=2N+1
\end{aligned}
\right.
\]
\end{lemma}
\begin{proof}
For $j=1,\dots,N$, using \eqref{tang-a-0},
\begin{align*}
|\omega(I_{2}, z_{j,0})| = & \Big| \Im \int_{\R^{N}}\, \frac i \eps \Big( \eps^{\beta}\tilde x \cdot \vv(t) - \mathcal{R}_{V}(t,\tilde x)\Big) \, \tilde w(t,\tilde x) \, \overline{(- \partial_{j}U(\tilde x))}\, d\tilde x \Big|\le\\
\le & \eps^{\beta-1} \int_{\R^{N}}\, |\tilde x \cdot \vv(t)| \, |\tilde w(t,\tilde x)| \left| \partial_{j} U (\tilde x)\right|\, d\tilde x + \eps^{-1} \int_{\R^{N}}\, \left|\mathcal{R}_{V}(t,\tilde x)\right| \, |\tilde w(t,\tilde x)| \left| \partial_{j} U (\tilde x)\right|\, d\tilde x 
\end{align*}
and then use $|\tilde x \cdot \vv(t)|\le |\tilde x| |\vv(t)|$ and Cauchy-Schwarz inequality.

The cases $j=N+1,\dots, 2N+1$ are proved in the same way.
\end{proof}

\begin{lemma}\label{i3} 
\[
\omega(I_{3}, z_{j,0}) = \left\{ 
\begin{aligned}
& 0\, , & j=1,\dots,N\\
& \eps^{1-2\beta}\, \omega(\tilde w, z_{j-N,0}) \, , & j=N+1,\dots, 2N\\
& 0\, , & j=2N+1
\end{aligned}
\right.
\]
\end{lemma}
\begin{proof}
Notice that
\[
I_{3} = i \eps^{1-2\beta}\, \mathcal{L}(\tilde w)
\]
where $\mathcal{L}$ is the Hessian of the energy associated to \eqref{ellittica}. Then
\[
\omega(I_{3}, z_{j,0}) = - \eps^{1-2\beta} \Im \int_{\R^{N}}\, \tilde w(t,\tilde x)\, \overline{i \mathcal{L}(z_{j,0})}\, d\tilde x
\]
and we conclude using Lemma 2 in \cite{colliding}.
\end{proof}

\begin{lemma}\label{i4} If $\norm{\tilde w(t,\cdot)}_{L^{2}}\le 1$ then
\[
|\omega(I_{4}, z_{j,0})| \le \eps^{1-2\beta}\, C(z_{j,0},U) \, \norm{\tilde w}_{L^{2}}
\]
for all $j=1,\dots,2N+1$.
\end{lemma}
\begin{proof}
It follows immediately from (N2) and (C2).
\end{proof}

We can now prove
\begin{proposition}\label{crescita}
Under the same assumptions of Theorem \ref{main-result}, it holds
\[
\max_{j=1,\dots,2N+1}\, \Big| \partial_{t}\, \omega(w,z_{j,\sigma}^{\eps}) \Big| = O(\eps^{\delta-1})
\]
for all $t\in (0,\tau)$ with $\tau= O(\eps^{-\delta})$.
\end{proposition}

\begin{proof}
From Lemma \ref{sympl-w-tilde-w}, we have for $j=1,\dots,N$
\[
\partial_{t}\, \omega(w,z_{j,\sigma}^{\eps}) = \eps^{2\gamma+\beta(N-1)}\, \omega(\partial_{t} \tilde w, z_{j,0}) - \frac 12 \eps^{2\gamma+\beta N-1}\, \xi_{j}(t) \omega(\partial_{t} \tilde w, z_{2N+1,0}) - \frac 12 \eps^{2\gamma+\beta N-1}\, \dot \xi_{j}(t) \omega(\tilde w, z_{2N+1,0})
\]
and for $j=N+1,\dots,2N$
\begin{align*}
\partial_{t}\, \omega(w,z_{j,\sigma}^{\eps}) =\, & \eps^{2\gamma+\beta(N+1)-1}\, \omega(\partial_{t} \tilde w, z_{j,0}) + \frac 12 \eps^{2\gamma+\beta N-1}\, a_{j-N}(t) \omega(\partial_{t} \tilde w, z_{2N+1,0}) + \\ 
+\, & \frac 12 \eps^{2\gamma+\beta N-1}\, \dot a_{j-N}(t) \omega(\tilde w, z_{2N+1,0})
\end{align*}

Hence using Lemmas \ref{i1}-\ref{i4}, we have the following estimates: for $j=1,\dots,N$
\begin{align*}
\Big| \partial_{t}\, \omega(w,z_{j,\sigma}^{\eps}) \Big| \le\, & \frac 12 \eps^{2\gamma + \beta N -1} \, \rho |\vv(t)| + \eps^{2\gamma + \beta (N-1) -1} \, \norm{\mathcal{R}_{V}(t,\cdot) U}_{L^{2}}\, \norm{\nabla U}_{L^{2}} +\\
+\, & \norm{\tilde w}_{L^{2}}\, \Big( \eps^{2\gamma + \beta N -1} \norm{|x| \left|\nabla U(x)\right|}_{L^{2}}\, |\vv(t)| + \eps^{2\gamma + \beta (N-1) -1} \norm{\mathcal{R}_{V}(t,\cdot)\, |\nabla U|}_{L^{2}}\Big) +\\
+\, & \eps^{2\gamma+\beta(N-3)+1}\, C(z_{j,0},U) \norm{\tilde w}_{L^{2}}+ \frac 12 \eps^{2\gamma+\beta(N-2)}\, |\xi(t)|\, C(z_{2N+1,0},U) \norm{\tilde w}_{L^{2}}+ \\
+\, & \frac 12 |\xi(t)|\, \norm{\tilde w}_{L^{2}}\, \Big( \frac 12 \eps^{2\gamma + \beta(N+1)-2} \norm{|x| U(x)}_{L^{2}}\, |\vv(t)| + \eps^{2\gamma + \beta N -2} \norm{\mathcal{R}_{V}(t,\cdot)\, U}_{L^{2}}\Big) + \\
+\, & \frac 12 \eps^{2\gamma + \beta N -1}\, \rho^{\frac 12}\, |\dot \xi(t)|\, \norm{\tilde w}_{L^{2}}\\
\end{align*}
for $j=N+1,\dots,2N$
\begin{align*}
\Big| \partial_{t}\, \omega(w,z_{j,\sigma}^{\eps}) \Big| \le\, & \norm{\tilde w}_{L^{2}}\, \Big( \frac 12 \eps^{2\gamma + \beta(N+2)-2} \norm{|x|^{2} U(x)}_{L^{2}}\, |\vv(t)| + \eps^{2\gamma + \beta(N+1)-2} \norm{\mathcal{R}_{V}(t,x)\, |x| U(x)}_{L^{2}}\Big) +\\
+\, & \eps^{2\gamma + \beta (N-1)}\, \norm{|x| U(x)}_{L^{2}}\, \norm{\tilde w}_{L^{2}} + \eps^{2\gamma+\beta(N-1)}\, C(z_{j,0},U) \norm{\tilde w}_{L^{2}}+\\
+\, & \frac 12 |a(t)|\, \norm{\tilde w}_{L^{2}}\, \Big( \frac 12 \eps^{2\gamma + \beta(N+1)-2} \norm{|x| U(x)}_{L^{2}}\, |\vv(t)| + \eps^{2\gamma + \beta N -2} \norm{\mathcal{R}_{V}(t,\cdot)\, U}_{L^{2}}\Big) + \\
+\, & \frac 12 \eps^{2\gamma+\beta(N-2)}\, |a(t)|\, C(z_{2N+1,0},U) \norm{\tilde w}_{L^{2}}+ \frac 12 \eps^{2\gamma + \beta N -1}\, \rho^{\frac 12}\, |\dot a(t)|\, \norm{\tilde w}_{L^{2}}\\
\end{align*}
Moreover we have from the proof of Theorem \ref{main-part-1} that
\begin{align*}
\Big| \partial_{t}\, \omega(w,z_{2N+1,\sigma}^{\eps}) \Big| \le\, & \norm{\tilde w}_{L^{2}} \Big( \frac 12 \eps^{2\gamma + \beta(N+1)-1} \norm{|x| U(x)}_{L^{2}}\, |\vv(t)| + \eps^{2\gamma + \beta N-2} \norm{|\mathcal{R}_{V}(t,\cdot)|\, U}_{L^{2}} +\\
+\, & \eps^{2\gamma + \beta(N-2)}\, C(z_{j,0},U)\Big)
\end{align*}

Arguing now as in the proof of Theorem \ref{main-result}, and using
\[
\norm{\tilde w}_{L^{2}} = \eps^{-\gamma -\beta \frac N2}\, \norm{w}_{L^{2}}
\]
we have for $\beta \ge 1$
\[
\max_{j=1,\dots,2N+1}\, \Big| \partial_{t}\, \omega(w,z_{j,\sigma}^{\eps}) \Big| \le C(U,a_{0},\xi_{0},V)\, \Big( \eps^{2\gamma + \beta(N-1)-1} + \norm{w}_{L^{2}} \eps^{\gamma + \beta(\frac N2-2)}\Big)
\]
By \eqref{stima-princ-2}, this implies that
\[
\max_{j=1,\dots,2N+1}\, \Big| \partial_{t}\, \omega(w,z_{j,\sigma}^{\eps}) \Big| \le C(U,a_{0},\xi_{0},V)\, \eps^{\gamma + \beta(\frac N2-2)}
\]
for all $t\in (0,\tau)$ with $\tau= O(\eps^{-\delta})$. 
\end{proof}


\end{document}